\documentclass[twoside,10pt,leqno, A4]{amsart}
\usepackage{amsfonts}
\usepackage{amsmath}
\usepackage{amscd}
\usepackage{amssymb}
\usepackage{amsthm}
\usepackage{amsrefs}
\usepackage{latexsym}
\usepackage{mathrsfs}
\usepackage{bbm}
\usepackage{amscd}
\usepackage{amssymb}
\usepackage{amsthm}
\usepackage{amsrefs}
\usepackage{latexsym}
\usepackage{mathrsfs}
\usepackage{bbm}
\usepackage{enumerate}
\usepackage{graphicx}
\usepackage{color}
\begin{document}
\newtheorem{theorem}[subsection]{Theorem}
\newtheorem{proposition}[subsection]{Proposition}
\newtheorem{lemma}[subsection]{Lemma}
\newtheorem{corollary}[subsection]{Corollary}
\newtheorem{conjecture}[subsection]{Conjecture}
\newtheorem{prop}[subsection]{Proposition}
\numberwithin{equation}{section}
\renewcommand{\thefootnote}{\fnsymbol{footnote}}
\newcommand{\dif}{\mathrm{d}}
\newcommand{\intz}{\mathbb{Z}}
\newcommand{\ratq}{\mathbb{Q}}
\newcommand{\natn}{\mathbb{N}}
\newcommand{\comc}{\mathbb{C}}
\newcommand{\rear}{\mathbb{R}}
\newcommand{\prip}{\mathbb{P}}
\newcommand{\uph}{\mathbb{H}}
\newcommand{\sumstar}{\sideset{}{^*}\sum}

\title{Zero Density Theorems for Families of Dirichlet $L$-functions}

\author[C. C. Corrigan]{Chandler C. Corrigan} 
\address{School of Mathematics and Statistics, University of New South Wales, Sydney, NSW 2052, Australia}
\email{c.corrigan@student.unsw.edu.au}

\author[L. Zhao]{Liangyi Zhao} 
\address{School of Mathematics and Statistics, University of New South Wales, Sydney, NSW 2052, Australia}
\email{l.zhao@unsw.edu.au}

\date{\today}
\maketitle

\begin{abstract}
In this paper, we prove some zero density theorems for certain families of Dirichlet $L$-functions.  More specifically, the subjects of our interest are the collections of Dirichlet $L$-functions associated with characters to moduli from certain sparse sets and of certain fixed orders.
\end{abstract}

\section{Introduction}

It goes without saying that the locations of the non-trivial zeros of Dirichlet $L$-functions are of fundamental importance in analytic number theory.  Let $\chi$ be a Dirichlet character of conductor $q$.  Suppose that $\rho=\beta + i \gamma$ with $\beta$, $\gamma \in \rear$ is a non-trivial zero of the Dirichlet $L$-function $L(s,\chi)$.  Let $\sigma > 1/2$ and $T >0$.  Set 
\[ N(\sigma, T, \chi) = \# \{ \rho: L(\rho,\chi)=0, \; \beta \geq \sigma, \; |\gamma| \leq T \} . \]
The generalized Riemann hypothesis (GRH) asserts that $\beta = 1/2$ for all $\rho$, i.e., $N(\sigma, T, \chi) = 0$ for all $\sigma > 1/2$ and $T>0$. \newline

Although GRH is currently still an unresolved conjecture, there have been many upper bounds over the past century for $N(\sigma, T, \chi)$ in the literature, both individually and on average as $\chi$ runs over a family of characters.  We refer the reader to \cite[Chapter 10]{HIEK} and \cite[Chapter 12]{HM} for discussions of these results.  In brief, these estimates, dubbed zero density theorems, amount to saying that the zeros lying off the critical line should at least be very rare. \newline

The aim of this paper is to extend these zero density results to various special collections of Dirichlet characters, more specifically, families of primitive Dirichlet characters to moduli from certain sparse sets and of certain fixed orders.  \newline

Our first result is on sparse sets of moduli.  Let $\mathcal{Q}$ be a set of natural numbers contained in $(Q_0, Q_0+Q]$.  Using the nomenclature of \cite{Ba1}, we define, for each $t \in \natn$, the set
\begin{equation*}
    \mathcal{Q}_t=\{q\in\mathbb{N}:tq\in\mathcal{Q}\} .
\end{equation*}
Suppose that for $t \in \natn$ and $0 \leq Q_0 \leq Q$, there is a $\Phi\geq 1$ such that the bound
\begin{equation}\label{atukl}
    \max_{Q_0/t\leq v\leq (Q_0+Q)/t} | \{q\in\mathcal{Q}_t\cap(v,v+u]:q\equiv l\text{ mod }k\} | \leq\left(1+\frac{|\mathcal{Q}_t|tu}{Qk}\right)\Phi
\end{equation}
holds for $(k,l)=1$.  In this case, we say that the set $\mathcal{Q}$ is well-distributed.  We now state our results for sparse sets of moduli.

\begin{theorem} \label{sparse}  
Let $T >1$, $\varepsilon >0$ and $\mathcal{Q} \subset (Q_0,Q_0+Q]$, with $|\mathcal{Q}|\leq Q^{1/2}$, be a well-distributed set of natural numbers such that \eqref{atukl} holds with $\Phi \ll (QT)^{\varepsilon}$.  Then for sufficiently large $T$ and any $\tfrac{1}{2}\leq \sigma\leq 1$, we have
\begin{equation*}
    \sum_{q\in\mathcal{Q}} \ \sumstar_{\chi\bmod{q}}N(\sigma,T,\chi)\ll (QT)^\varepsilon\min\left(\eta_{Q,T}, |\mathcal{Q}|(QT)^{3(1-\sigma)/(2-\sigma)},\left(|\mathcal{Q}|Q^3T^2\right)^{(1-\sigma)/\sigma}\right)
\end{equation*}
where  
\begin{equation*}
    \eta_{Q,T} = T^{3(1-\sigma)/(2-\sigma)}\begin{cases}|\mathcal{Q}|^{3(3-4\sigma)/(5-4\sigma)}Q&\text{if}\quad\tfrac{1}{2}\leq\sigma\leq\tfrac{3}{4}\\ \left(|\mathcal{Q}|^{4\sigma-3}Q^{12\sigma-7}\right)^{(1-\sigma)/(9\sigma-4(\sigma^2+1))}&\text{otherwise.}\end{cases}
\end{equation*}
Here the implied constant depends on $\varepsilon$ alone.
\end{theorem}

As in \cite{Ba1}, one can easily check that the set of perfect $k$-powers, with $k \geq 2$, form a well-distributed sparse set.  Thus we readily get the following corollary from Theorem~\ref{sparse}.

\begin{corollary}\label{bzkpower}
For $k\geq3$, sufficiently large $Q$, $T>0$, and any $\varepsilon>0$, we have
\begin{equation*}
    \sum_{q\leq Q} \ \sumstar_{\chi\text{ mod }q^k}N(\sigma,T,\chi)\ll (QT)^\varepsilon\min\left(\left(Q^{3k+2-(3k+1)\sigma}T^{3(1-\sigma)}\right)^{1/(2-\sigma)},\left(Q^{3k+1}T^2\right)^{(1-\sigma)/\sigma}\right),
\end{equation*}
where the implied constant depends on $\varepsilon$ and $k$ at most.
\end{corollary}

Corollary~\ref{bzkpower} also holds for $k=2$.  But for square moduli, we have the following which is better.

\begin{theorem}\label{bz2}
For sufficiently large $Q$, $T>0$ and any $\varepsilon>0$, we have
\begin{equation*}
    \sum_{q\leq Q} \ \sumstar_{\chi\bmod{q^2}}N(\sigma,T,\chi)\ll(QT)^\varepsilon \min\left(\eta_{Q,T},\left(Q^7T^2\right)^{(1-\sigma)/\sigma}\right)
\end{equation*}
where
\begin{equation*}
\eta_{Q,T}=\begin{cases}Q^{(17-16\sigma)/2(2-\sigma)}T^{3(1-\sigma)/(2-\sigma)}&\text{if}\quad\tfrac{1}{2}\leq\sigma\leq\tfrac{3}{4}\\ Q^{(1-\sigma)(28\sigma-17)/(9\sigma-4(\sigma^2+1))}T^{3(1-\sigma)/(2-\sigma)}&\text{otherwise,}
    \end{cases}
\end{equation*}
and the implied constant depends on $\varepsilon$ alone.
\end{theorem}

Our result on fixed order characters is as follows.

\begin{theorem} \label{fixedorder}
Let $j \in \{2, 3, 4, 6\}$ and $\mathcal{C}_j(Q)$ be the collection of primitive Dirichlet characters of order $j$ and conductor $q \leq Q$.  We have for $T\gg1$ that
\begin{equation} \label{quadzeroden}
\sum_{\chi \in \mathcal{C}_2(Q)} N(\sigma, T, \chi) \ll (QT)^{\varepsilon} \min \left( (Q^3T^4)^{(1-\sigma)/(2-\sigma)} , (QT)^{3(1-\sigma)/\sigma} \right) ,
\end{equation}
for $T\gg Q^\frac{2}{3}$ that
\begin{equation} \label{CSzeroden}
\sum_{\chi \in \mathcal{C}_j(Q)} N(\sigma, T, \chi) \ll (QT)^{\varepsilon} \min \left( Q^{(125-108\sigma)/(90-72\sigma)}T^{(49-44\sigma)/(22-8\sigma)} , (QT)^{7(1-\sigma)/2\sigma} \right) , \quad \mbox{for} \; j =3, 6,
\end{equation}
and for $T\gg Q^\frac{1}{2}$ that
\begin{equation} \label{quartzeroden}
\sum_{\chi \in \mathcal{C}_4(Q)} N(\sigma, T, \chi) \ll (QT)^{\varepsilon} \min \left( Q^{(41-36\sigma)/(30-24\sigma)}T^{(49-44\sigma)/(22-8\sigma)} , (QT)^{7(1-\sigma)/2\sigma} \right) .
\end{equation}
Here the implied constants above depend on $\varepsilon$.
\end{theorem}

\section{The setup}

Our plan of attack goes along similar lines as those in \cite[Chapter 12]{HM}.  Let $\alpha>0$ be some fixed constant and $\mathcal{C}$ a family of primitive Dirichlet characters none of which have conductors greater than $Q$.  Now define $\mathcal{R}$ to be a finite set of $(\rho,\chi)$ such that $L(\rho,\chi)=0$ for some $\chi \in \mathcal{C}$, where $\beta\geq\sigma>\frac{1}{2}$ and $|\gamma|\leq T$ for all $(\rho,\chi)\in\mathcal{R}$ and $|\gamma-\gamma'|\geq\alpha\log QT$ for some constant $\alpha$ and all distinct $(\rho,\chi)$ and $(\rho',\chi)\in\mathcal{R}$. \newline

Let $\{ a_n \}$ be a arbitrary sequence of complex numbers.  We define
\begin{equation*} \label{RSdef}
    R(\chi)=\sum_{n\leq N}a_n\chi(n) \quad\text{and}\quad S(s,\chi)=\sum_{n\leq N} \frac{a_n \chi(n)}{n^s}.
\end{equation*}
If $\{ A_l \}$ and $\{ B_l \}$ are sequences of non-negative real numbers and $L \in \natn$, then we set
\[ \Delta (Q,N) = \sum_{l \leq L} Q^{A_l} N^{B_l} \quad \mbox{and} \quad \Delta_T (Q,N) = \sum_{l \leq L} Q^{A_l} N^{B_l} T^{1-B_l}. \]

Under the above notations and conditions, we can show, using the same arguments as those for (12.28) or (12.29) in \cite{HM}, that it is possible to choose the elements of $\mathcal{R}$ so that for any $\varepsilon>0$
\begin{equation} \label{montgomerybound}
\sum_{\chi\in \mathcal{C}} N(\sigma,T,\chi) \ll (QT)^\varepsilon\left(|\mathcal{R}|+1\right),
\end{equation}
where the implied constant depends on $\varepsilon$ only.  Consequently, our attention is shifted to estimating the size of $\mathcal{R}$. \newline

We define for $X\geq2$
\begin{equation*}
M_X(s,\chi)=\sum_{n\leq X} \frac{\mu(n)\chi(n)}{n^s},
\end{equation*}
where $\mu$ is the M\"obius function.  We note here that the Dirichlet series of $\mathfrak{M}_X(s,\chi)=L(s,\chi)M_X(s,\chi)$ has coefficients $\mathfrak{m}_{X,n}\chi(n)$ with
\begin{equation*}
    \mathfrak{m}_{X,n}=\sum_{\substack{d|n\\d\leq X}}\mu(d).
\end{equation*}
Thus $\mathfrak{m}_{X,1}=1$, $\mathfrak{m}_{X,n}=0$ for $2\leq n\leq X$, and $|\mathfrak{m}_{X,n}|\leq \tau(n)$ for $n>X$, with $\tau$ denoting the divisor function. \newline

Now we consider the Dirichlet series with coefficients $\mathfrak{m}_{X,n}\chi(n)e^{-n/Y}$ where $1\ll X\ll Y\ll(QT)^K$ for some sufficiently large $K\geq1$.  We have, from (12.25) and (12.26) of \cite{HM}, that for sufficiently large $\alpha=3A$, each $(\rho,\chi)\in\mathcal{R}$ satisfies at least one of the inequalities
\begin{equation} \label{mon12.25}
    \left|\sum_{X<n\leq Y^2}\mathfrak{m}_{X,n}\chi(n)n^{-\rho}e^{-n/Y}\right|\geq\frac{1}{6}
\end{equation}
or
\begin{equation} \label{mon12.26}
    \frac{1}{2\pi}\left|\int\limits_{-A\log (QT)}^{A\log (QT)}\mathfrak{M}_X\left(\tfrac{1}{2}+i\gamma+iu,\chi\right)Y^{1/2-\beta+iu}\Gamma\left(\tfrac{1}{2}-\beta+iu\right) \dif u\right|\geq\frac{1}{6}.
\end{equation}
Let $\mathcal{R}_1$ and $\mathcal{R}_2$ be the sets consisting of all elements of $\mathcal{R}$ satisfying \eqref{mon12.25} and \eqref{mon12.26}, respectively.  Hence
\begin{equation} \label{RR1R2}
 | \mathcal{R} | \leq |\mathcal{R}_1| + |\mathcal{R}_2|
\end{equation}
and it suffices to estimate from above the sizes of $\mathcal{R}_1$ and $\mathcal{R}_2$. \newline

Along similar lines as the treatment in \cite{HM}, we obtain
\begin{equation} \label{sumu2u}
| \mathcal{R}_1 | \ll(\log Y)^3 \sum_{(\rho,\chi)\in\mathcal{R}_1}\left|\sum_{n=U}^{2U}\mathfrak{m}_{X,n}\chi(n)n^{-\rho}e^{-n/Y}\right|^2,
\end{equation}
for some $U$ with $X \leq U\leq Y^2$.  For $\mathcal{R}_2$, we get
\begin{align} \label{monhol}
| \mathcal{R}_2 | &\ll Y^{2/3-4\sigma/3} (QT)^\varepsilon\left(\sum_{(\rho,\chi)\in\mathcal{R}_2}\left|L\left(\tfrac{1}{2}+it_\rho,\chi\right)\right|^4\right)^{1/3} \left(\sum_{(\rho,\chi)\in\mathcal{R}_2}\left|M_X\left(\tfrac{1}{2}+it_\rho,\chi\right)\right|^2\right)^{2/3}.
\end{align}
Here for each $(\rho,\chi) \in \mathcal{R}_2$, $t_{\rho}$ is defined to be the real number in the interval $[\gamma - A \log (QT), \gamma+A \log (QT)]$ for which $|\mathfrak{M}_X(\frac{1}{2}+it_{\rho},\chi)|$ is maximum. \newline

Now we are led to estimate sums of the form,
\begin{equation} \label{Save}
\sum_{\chi\in \mathcal{C}} \sum_{s\in\mathcal{S}_\chi}\bigg|\sum_{n\leq N} \frac{a_n\chi(n)}{n^s}\bigg|^2,
\end{equation}
where $\mathcal{S}_\chi$ is a set of complex numbers.  To that end, various kinds of large sieve inequalities will play an indispensable role.  We refer the reader to \cites{Ramare, HM} and \cite[Chapter 7]{HIEK} for more extensive discussions on the large sieve, a subject of independent interest. \newline

We first write down a general result for sums of the form \eqref{Save}.  

\begin{lemma} \label{Savebound}
Let $\mathcal{C}$ be an arbitrary set of primitive Dirichlet characters with conductors at most $Q$ and $\mathcal{S}_\chi$ be a finite set of complex numbers $s=\sigma+it$.  Suppose $T_0$, $T$, $\sigma_0>\delta>0$ are such that $T_0+\delta/2 \leq|t|\leq T_0+T-\delta/2$ for all $s\in\mathcal{S}_\chi$, $1/2 \leq \sigma_0 \leq \sigma \leq 1$ for all $s\in\mathcal{S}_\chi$ , and $|t-t'|\geq\delta$ for distinct $s$, $s'\in\mathcal{S}_\chi$.  If the bound
\[ \sum_{\chi \in \mathcal{C}} | R(\chi)|^2 \ll \Delta (Q,N) \sum_{n \leq N} |a_n|^2 \]
 holds, then we have
 \[ \sum_{\chi \in \mathcal{C}} \sum_{s \in \mathcal{S}_\chi} | S(s, \chi)|^2 \ll \left( \frac{1}{\delta} + \log N \right) \Delta_T(Q,N)  \sum_{n \leq N} \frac{|a_n|^2}{ n^{2\sigma_0}} \left(1+\log\frac{\log2N}{\log2n}\right) . \]
\end{lemma}

\begin{proof}
The proof is rather standard and thus we only give a sketch here.  Let 
\[ S_u(s,\chi) = \sum_{2 \leq n \leq u} a_n \chi(n) n^{-s}. \]
Partial summation and Cauchy's inequality give
\[ |S_u(s,\chi)|^2 \ll |a_1|^2 + |S(\sigma_0+it, \chi)|^2 + \int\limits_2^N |S_u(\sigma_0+it , \chi)|^2 \frac{\dif u}{u \log u} . \]
Using \cite[Lemma 1.4]{HM}, we get
\begin{equation} \label{Sitbound1}
 \sum_{\chi \in \mathcal{C}} \sum_{t \in \mathcal{T}_\chi} | S(it, \chi)|^2 \ll \left( \frac{1}{\delta} + \log N \right) \sum_{\chi \in \mathcal{C}} \int\limits_{T_0}^{T_0+T} |S(it, \chi)|^2 \dif t  
 \end{equation}
where $\mathcal{T}_{\chi} = \{ t : s = \sigma +it \in \mathcal{S}_{\chi} \}$.  Now arguing along similar lines as the proof of \cite[Theorem 2]{Galla}, we arrive at
\begin{equation} \label{Sitbound2}
 \sum_{\chi \in \mathcal{C}} \int\limits_{T_0}^{T_0+T} |S(it, \chi)|^2 \dif t \ll \Delta_T(Q,N) \sum_{n \leq N} |a_n|^2.
 \end{equation}
The desired bound follows easily by combining all the bounds above.
\end{proof}

Thus from our discussion above, we have the following general result which can be used for deriving a zero density result for any collection of primitive Dirichlet characters, if the corresponding large sieve inequality and bound for the fourth moment of $L$-functions are available.

\begin{theorem}\label{mainresult1}
Let $\mathcal{C}$ be a finite family of primitive Dirichlet characters, none of which have conductors greater than $Q$, and suppose that
\begin{equation*}
    \sum_{\chi\in\mathcal{C}}|R(\chi)|^2\ll\Delta(Q,N)\sum_{n\leq N}|a_n|^2\quad\text{and}\quad\sum_{\chi\in\mathcal{C}}\int\limits_{-T}^{T}\left|L\left(\tfrac{1}{2}+it,\chi\right)\right|^4\dif t\ll\mathfrak{L}
\end{equation*}
hold.  Then for any $\tfrac{1}{2} \leq \sigma \leq 1$, and $X,Y$ satisfying $1\ll X\ll Y\ll (QT)^K$ for some absolute constant $K$, there is a $U$ with $X\ll U\ll Y^2$ such that
\begin{equation*}
    \sum_{\chi\in\mathcal{C}}N(\sigma,T,\chi)\ll (QT)^\varepsilon\left(\mathfrak{L}^{1/3} \Delta_T(Q,X)^{2/3}Y^{2(1-2\sigma)/3}+\Delta_T(Q,U)U^{1-2\sigma}e^{-2U/Y}\right),    
\end{equation*}
where the implied constant depends on $\varepsilon$ alone.
\end{theorem}

\begin{proof}
We take $\delta=3A\log QT$ in Lemma~\ref{Savebound}, where $A$ is as in \eqref{mon12.26}, and obtain 
\begin{equation}\label{thmr1}
    \sum_{(\rho,\chi)\in\mathcal{R}_1}\Big|\sum_{n=U}^{2U}\mathfrak{m}_{X,n}\chi(n)n^{-\rho}e^{-n/Y}\Big|^2\ll(QT)^\varepsilon\Delta_T(Q,U)U^{1-2\sigma}e^{-2U/Y},
\end{equation}
and
\begin{equation}\label{mxthm}
    \sum_{(\rho,\chi)\in\mathcal{R}_2}\left|M_X\left(\tfrac{1}{2}+it,\chi\right)\right|^2\ll(QT)^\varepsilon\Delta_T(Q,X).
\end{equation}
Using similar methods to \cite[Theorem 10.3]{HM} we can show that
\begin{equation}\label{ellbound}
    \sum_{(\rho,\chi)\in\mathcal{R}_2}\left|L\left(\tfrac{1}{2}+it,\chi\right)\right|^4\ll(QT)^\varepsilon\mathfrak{L}.
\end{equation}
Now, from \eqref{sumu2u} and \eqref{thmr1} we obtain a bound for $|\mathcal{R}_1|$ and \eqref{monhol}, \eqref{mxthm}, and \eqref{ellbound} give rise to a majorant for $|\mathcal{R}_2|$.  The result now follows from \eqref{montgomerybound} and \eqref{RR1R2}.
\end{proof}

Our second general result below does not require any large sieve type bound.

\begin{theorem}\label{mainresult2}  Let $\mathcal{C},Q,T,\mathfrak{L}$ be as in Theorem~\ref{mainresult1}.  Then for any $\tfrac{1}{2}<\sigma\leq1$ and any $\varepsilon>0$ we have
\begin{equation*}
    \sum_{\chi\in\mathcal{C}}N(\sigma,T,\chi)\ll(QT)^\varepsilon\left(\left(\mathfrak{L}Q^2T\right)^{(1-\sigma)/\sigma}+\left(Q^2T\right)^{(1-\sigma)/(2\sigma-1)}\right),
\end{equation*}
where the implied constant depends on $\varepsilon$ alone.
\end{theorem}

\begin{proof}
The proof follows the same arguments as \cite[(12.14)]{HM}.  The only difference is that we do not insert any specific bound for the fourth moment of $L$-functions.
\end{proof}

\section{Proof of Theorems~\ref{sparse} and~\ref{bz2}}

Before proving Theorems~\ref{sparse} and~\ref{bz2}, we need the following Lemma.

\begin{lemma}  \label{L4mom}
Let $\mathcal{Q}$ be as above.  Then for any $T\geqslant2$ and $\varepsilon>0$ we have
\begin{align*}
    \sum_{q\in\mathcal{Q}} \ \sumstar_{\chi\bmod{q}}\int\limits_{-T}^{T}\left|L\left(\tfrac{1}{2}+it,\chi\right)\right|^4\:\dif t\ll|\mathcal{Q}|(QT)^{1+\varepsilon},
\end{align*}
where the implied constant depends on $\varepsilon$ alone.
\end{lemma}

\begin{proof}
This result follows readily from \cite[Theorem 10.1]{HM}.
\end{proof}

Now we proceed with the proof of Theorem~\ref{sparse}.

\begin{proof}[Proof of Theorem~\ref{sparse}] From \cite[Theorem 2]{Ba1}, we have 
\begin{equation} \label{baiersparse}
    \sum_{q\in\mathcal{Q}} \ \sumstar_{\chi \bmod{q}}|R(\chi)|^2\ll(QN)^\varepsilon\left(|\mathcal{Q}|Q+N+QN^{1/2}\right)\sum_{n\leq N}|a_n|^2.
\end{equation}
Moreover, the classical large sieve inequality gives (see the discussion around (5.4) and (5.5) in \cite{LZ1})
\begin{equation}\label{trivialsparse}
    \sum_{q\in\mathcal{Q}} \ \sumstar_{\chi \bmod{q}}|R(\chi)|^2\ll \min\left(|\mathcal{Q}|(Q+N),Q^2+N\right)\sum_{n\leq N}|a_n|^2.
\end{equation}
Using \eqref{baiersparse}, Lemma~\ref{L4mom} and Theorem~\ref{mainresult1}, we obtain
\begin{equation*}
\begin{split}
    \sum_{q\in\mathcal{Q}} & \ \sumstar_{\chi\bmod{q}}N(\sigma,T,\chi) \\
    & \ll(QT)^\varepsilon\left((|\mathcal{Q}|QT)^{1/3}\left(|\mathcal{Q}|QT+X+QT^{1/2}X^{1/2}\right)^{2/3}Y^{2(1-2\sigma)/3}+|\mathcal{Q}|QTX^{1-2\sigma}+Y^{2-2\sigma} \right. \\
     & \hspace*{5cm}  \left. +QT^{1/2}\begin{cases}Y^{3/2-2\sigma}, &\text{if }\tfrac{1}{2}\leq\sigma\leq\tfrac{3}{4}\\X^{3/2-2\sigma} , &\text{otherwise}\end{cases}\right).
    \end{split}
\end{equation*}
On taking 
\begin{equation*}
    X=|\mathcal{Q}|^2T\quad\text{and}\quad Y=|\mathcal{Q}|^{6/(5-4\sigma)}T^{3/2(2-\sigma)}
\end{equation*}
in the case $1/2 \leq \sigma \leq 3/4$, and
\[  X=|\mathcal{Q}|^{(2\sigma-2)/(9\sigma-4(\sigma^2+1))}Q^{(4\sigma-2)/(9\sigma-4(\sigma^2+1))}T \]
 and
 \[ Y=|\mathcal{Q}|^{(4\sigma-3)/(18\sigma-8(\sigma^2+1))}Q^{(12\sigma-7)/(18\sigma-8(\sigma^2+1))}T^{3/2(2-\sigma)} \]
in the the other case, we obtain
\begin{equation}\label{cor1}
   \sum_{q\in\mathcal{Q}} \ \sumstar_{\chi\bmod{q}}N(\sigma,T,\chi)\ll(QT)^\varepsilon\begin{cases}|\mathcal{Q}|^{3(3-4\sigma)/(5-4\sigma)}QT^{3(1-\sigma)/(2-\sigma)}, &\text{if}\quad\tfrac{1}{2}\leq\sigma\leq\tfrac{3}{4}\\ \left(|\mathcal{Q}|^{4\sigma-3}Q^{12\sigma-7}\right)^{(1-\sigma)/(9\sigma-4(\sigma^2+1))}T^{3(1-\sigma)/(2-\sigma)}, &\text{otherwise.}\end{cases}
\end{equation}
Now, using \eqref{trivialsparse} and Lemma~\ref{L4mom} in Theorem~\ref{mainresult1}, we also have
\begin{equation*}
\begin{split}
    \sum_{q\in\mathcal{Q}} & \ \sumstar_{\chi\bmod{q}}N(\sigma,T,\chi) \\
    & \ll(QT)^\varepsilon\Big( (|\mathcal{Q}|QT)^{1/3}\min\left(|\mathcal{Q}|QT+|\mathcal{Q}|X,Q^2T+X\right)^{2/3}Y^{2(1-2\sigma)/3} \\
    & \hspace*{4cm} +\min\left(|\mathcal{Q}|QTX^{1-2\sigma}+|\mathcal{Q}|Y^{2-2\sigma},Q^2TX^{1-2\sigma}+Y^{2-2\sigma}\right)\Big).
    \end{split}
\end{equation*}
We firstly take 
\begin{equation*}
    X=QT\quad\text{and}\quad Y=X^{3/2(2-\sigma)}
\end{equation*}
and then we take
\begin{equation*}
    X=Q^2T\quad\text{and}\quad Y=X^{3/2(2-\sigma)},
\end{equation*}
and on comparing these results, we arrive at
\begin{equation}\label{cor2}
    \sum_{q\in\mathcal{Q}} \ \sumstar_{\chi\bmod{q}}N(\sigma,T,\chi)\ll(QT)^\varepsilon\min\left(|\mathcal{Q}|Q^{3(1-\sigma)/(2-\sigma)},Q^{6(1-\sigma)/(2-\sigma)}\right)T^{3(1-\sigma)/(2-\sigma)}.
\end{equation}
Our desired result follows from comparing \eqref{cor1} with \eqref{cor2} and Theorem~\ref{mainresult2}.
\end{proof}

Now, considering the case in which $\mathcal{Q}$ is the set of $k$-power moduli, then from \cite[Theorem 1]{LZ10}, we have, for any integer $k\geq3$, and any $Q$, $\varepsilon>0$, 
\begin{equation} \label{BZls}
    \sum_{q\leq Q} \ \sumstar_{\chi \bmod{q^k}}|R(\chi)|^2\ll(QN)^\varepsilon\left(Q^{k+1}+N+Q^kN^{1/2}\right)\sum_{n\leq N}|a_n|^2,
\end{equation}
which is just a special case of \eqref{baiersparse}.  Thus \eqref{BZls} will lead to a result already contained in Theorem~\ref{sparse} and gives precisely Corollary~\ref{bzkpower}.   We note here that \eqref{BZls} has been improved in certain ranges by a number of authors, K. Halupczok \cites{Halu2,Hal18}, M. Munsch \cite{mun21}, K. Halupczok and M. Munsch \cite{halmun22}, and R. C. Baker, M. Munsch, and I. E. Shparlinski \cite{BaMuSh}.  Unfortunately, using the method here, the results in \cites{Halu2, Hal18,halmun22,BaMuSh} do not lead to any outcome better than Corollary~\ref{bzkpower}. \newline

In the case of square moduli, the best known large sieve inequality is found in \cite{LZ11}, 
\begin{equation}\label{baierzhao2}
    \sum_{q\leq Q}\ \sumstar_{\chi\bmod{q^2}}|R(\chi)|^2\ll(QN)^\varepsilon\left(Q^3+N+\min\left(N\sqrt{Q},Q^2\sqrt{N}\right)\right)\sum_{n\leq N}|a_n|^2
\end{equation}
and from this we can derive Theorem~\ref{bz2}, which is better than what Theorem~\ref{sparse} gives in certain regions. 

\begin{proof}[Proof of Theorem~\ref{bz2}] By Theorem~\ref{mainresult1} we have 
\begin{equation*}
\begin{split}
    \sum_{q\leq Q} & \ \sumstar_{\chi\bmod{q^2}}N(\sigma,T,\chi) \\
    & \ll(QT)^\varepsilon \left( \left(Q^3T\right)^{1/3}\left(Q^3T+X+\min\left(Q^{1/2}X,Q^2T^{1/2}X^{1/2}\right)\right)^{2/3}Y^{2(1-2\sigma)/3}+Q^3TX^{1-2\sigma}+Y^{2-2\sigma} \right. \\
    & \hspace*{4cm} \left. +\min\left(Q^{1/2}Y^{2-2\sigma},Q^2\begin{cases}Y^{3/2-2\sigma}&\text{if}\quad\sigma\leq\tfrac{3}{4}\\X^{3/2-2\sigma}&\text{otherwise}\end{cases}\right) \right).
    \end{split}
\end{equation*}
To get the desired result, we simply take
\begin{equation*}
    X=Q^2T\quad\text{and}\quad Y=Q^{15/4(2-\sigma)}T^{3/2(2-\sigma)}
\end{equation*}
if $\sigma \leq \tfrac{3}{4}$, and
\begin{equation*}
    X=Q^{(10\sigma-6)/(9\sigma-4(\sigma^2+1))}T\quad\text{and}\quad Y=Q^{(28\sigma-17)/(18\sigma-8(\sigma^2+1))}T^{3/2(2-\sigma)}
\end{equation*}
in the latter case.  Hence
\begin{equation}\label{corsq}
    \sum_{q\leq Q} \ \sumstar_{\chi\bmod{q^2}}N(\sigma,T,\chi)\ll(QT)^\varepsilon \begin{cases}Q^{(17-16\sigma)/2(2-\sigma)}T^{3(1-\sigma)/(2-\sigma)}&\text{if}\quad\tfrac{1}{2}\leq\sigma\leq\tfrac{3}{4}\\ Q^{(1-\sigma)(28\sigma-17)/(9\sigma-4(\sigma^2+1))}T^{3(1-\sigma)/(2-\sigma)}&\text{otherwise.}
    \end{cases}
\end{equation}
The result follows on comparing \eqref{corsq} with Theorem~\ref{sparse}.
\end{proof}

In \cite{LZ1}, an optimal conjectural large sieve inequality for power moduli is given and it yields the bound
\begin{equation}     \label{conjecturek}
    \sum_{q\leq Q} \ \sumstar_{\chi \bmod{q^k}}|R(\chi)|^2\ll Q^\varepsilon\left(Q^{k+1}+N\right)\sum_{n\leq N}|a_n|^2.
\end{equation}
If \eqref{conjecturek} holds, then
\begin{equation*}
    \sum_{q\leq Q} \ \sumstar_{\chi \bmod{q^k}} N(\sigma,T,\chi)\ll\left(Q^{k+1}T\right)^{3(1-\sigma)/(2-\sigma)+\varepsilon}
\end{equation*}
holds for all positive $Q$ and $T$.

\section{Proof of Theorem~\ref{fixedorder}}

To establish Theorem~\ref{fixedorder}, we require, in view of \eqref{sumu2u} and \eqref{monhol}, bounds for the following
\begin{equation} \label{R&L4}
 \sum_{\chi \in \mathcal{C}} | R(\chi) |^2 \quad \mbox{and} \quad \sum_{\chi \in \mathcal{C}} | L ( \tfrac{1}{2} + it , \chi) |^4 ,
 \end{equation}
where $\mathcal{C}$ is the family of characters under our consideration. \newline

For $\mathcal{C}_2 (Q)$, using \cite[Corollary 3]{HB2}, we get
\begin{equation} \label{quadRbound}
\sum_{\chi \in \mathcal{C}_2(Q)} | R(\chi) |^2 \ll (QN)^{\varepsilon} (QN +N^2) \max_{n \leq N} |a_n|^2.
\end{equation}
By setting $\sigma=1/2$ in Theorem 2 of \cite{HB2}, we get for $T>1$ and $|t| \leq T$,
\begin{equation} \label{quadL4bound}
\sum_{\chi \in \mathcal{C}_2(Q)} | L ( \tfrac{1}{2} + it , \chi) |^4 \ll (QT)^{1+\varepsilon} ,
\end{equation}
Using \eqref{quadRbound} and Lemma~\ref{Savebound} with some minor changes (the bound \eqref{quadRbound} is formally different from what is in the condition of Lemma~\ref{Savebound}), we get
\begin{equation} \label{quadR1bound} \sum_{(\rho,\chi)\in\mathcal{R}_1}\left|\sum_{n=U}^{2U}\mathfrak{m}_{X,n}\chi(n)n^{-\rho}e^{-n/Y}\right|^2 \ll (QT)^{\varepsilon} (QTX^{1-2\sigma} +Y^{2-2\sigma} ) ,
\end{equation}
Now \eqref{Sitbound1} and \eqref{Sitbound2}, together with \cite[Corollary 1]{HB2}, produce the bound
\begin{equation} \label{quadMbound}
 \sum_{(\rho,\chi)\in\mathcal{R}_2}\left|M_X\left(\tfrac{1}{2}+it_\rho,\chi\right)\right|^2  \ll (QT)^{\varepsilon} (QT+X).
 \end{equation}
Now putting \eqref{quadL4bound}, \eqref{quadR1bound} and \eqref{quadMbound} into \eqref{sumu2u} and \eqref{monhol}, we get
\[  \sum_{\chi\in \mathcal{C}_2(Q)} N(\sigma, T, \chi) \ll (QT)^{\varepsilon} \left( (QT^2)^{1/3} (QT+X)^{2/3} Y^{2(1-2\sigma)/3} + QT X^{1-2\sigma} +Y^{2-2\sigma} \right). \]
Setting 
\[ X=QT \quad \mbox{and} \quad Y= (Q^3T^4)^{1/(2(2-\sigma))} , \]
we arrive at the first term in the minimum in \eqref{quadzeroden}. \newline

In the case of $\mathcal{C}_j (Q)$ with $j=3$, $4$ and $6$, we use the results in \cites{BaYo, GaoZhao2}.  The only minor obstruction is that one requires the sum over $n$ in $R(\chi)$ to be over square free $n$'s.  This is easily handled by rewriting $n=kl^2$ with $k$ square-free and applying Cauchy's inequality and then utilizing the large sieve inequalities for cubic, quartic and sextic characters.  For $j=3$ and $6$, we get from Theorems 1.4 and 1.5 of \cite{BaYo},
\begin{equation} \label{cubicRbound}
\sum_{\chi \in \mathcal{C}_j(Q)} | R(\chi) |^2 \ll (QN)^{\varepsilon} \min \{ Q^{5/2} N^{1/2} +N^{3/2}, Q^{11/9} + Q^{2/3} N \} \sum_{n \leq N} |a_n|^2.
\end{equation}
With $j=4$, from Lemma 2.10 of \cite{GaoZhao2}, which is an improvement of \cite[Theorem 1.2]{GaoZhao}, emerges the bound
\begin{equation} \label{quarticRbound}
\sum_{\chi \in \mathcal{C}_4(Q)} | R(\chi) |^2 \ll (QN)^{\varepsilon} \min\{ Q^{3/2} N^{1/2} +N^{3/2}, Q^{7/6} + Q^{2/3} N \} \sum_{n \leq N} |a_n|^2.
\end{equation}
The results in \cites{BaYo, GaoZhao2} have more terms in the minimum than those given in \eqref{cubicRbound} and \eqref{quarticRbound}.  Here, we only cite what we will use in the sequal. \newline

If $j=3$ or $6$, then for all $T\gg Q^\frac{2}{3}$ we have
\begin{equation} \label{CQSL4bound}
\sum_{\chi \in \mathcal{C}_j(Q)} | L ( \tfrac{1}{2} + it , \chi) |^4 \ll (QT)^{3/2+\varepsilon} . 
\end{equation}
The bound \eqref{CQSL4bound} also holds if $j=4$ and $T\gg Q^\frac{1}{2}$.  The proof of \eqref{CQSL4bound} uses the same arguments as that of \cite[Theorem 2]{HB2}.  The only difference is, instead of \eqref{quadRbound}, one uses \eqref{cubicRbound} or \eqref{quarticRbound} with the first terms in the minimums at the appropriate places. \newline

Now proceeding in the same way as for $\mathcal{C}_2(Q)$, using \eqref{CQSL4bound} and the second terms in the minimums given in the bounds \eqref{cubicRbound} and \eqref{quarticRbound}, we deduce
\begin{equation} \label{cubicxybound}
\begin{split}
\sum_{\chi \in \mathcal{C}_j(Q)} & N(\sigma, T, \chi) \\
&  \ll (QT)^{\varepsilon} \left( (Q^{3/2}T^{5/2})^{1/3} (Q^{11/9} T + Q^{2/3}X)^{2/3} Y^{2(1-2\sigma)/3} + Q^{11/9} T \max (X^{3/2-2\sigma}, Y^{3/2-2\sigma} ) + Q^{2/3}Y^{5/2-2\sigma} \right)
\end{split}
\end{equation}
for $j = 3$, $6$ and
\begin{equation} \label{quartxybound}
\begin{split}
\sum_{\chi \in \mathcal{C}_4(Q)} & N(\sigma, T, \chi) \\
&  \ll (QT)^{\varepsilon} \left( (Q^{3/2}T^{5/2})^{1/3} (Q^{7/6} T + Q^{2/3}X)^{2/3} Y^{2(1-2\sigma)/3} + Q^{7/6} T \max (X^{3/2-2\sigma}, Y^{3/2-2\sigma} ) + Q^{2/3}Y^{5/2-2\sigma} \right) .
\end{split}
\end{equation}
Taking
\[ X = Q^{5/27}{T^{1/2}} \quad \mbox{and} \quad Y = Q^{5/(45-36\sigma)} T^{12/(11-4\sigma)} \]
in \eqref{cubicxybound} and
\[ X = Q^{2/9}{T^{1/2}} \quad \mbox{and} \quad Y = Q^{2/(15-12\sigma)} T^{12/(11-4\sigma)} \]
in \eqref{quartxybound}, we get the first terms in the minimums in \eqref{CSzeroden} and \eqref{quartzeroden}.  The second terms in the minimums in \eqref{quadzeroden}, \eqref{CSzeroden} and \eqref{quartzeroden} are derived from Theorem~\ref{mainresult2} and either \eqref{quadL4bound} or \eqref{CQSL4bound}.  This concludes the proof of Theorem~\ref{fixedorder}. \newline

M. Jutila \cite[Theorem 2]{Jutila} previously gave the bound
\begin{equation}\label{jutrealbound}
\sum_{\chi\in\mathcal{C}_2(Q)}N(\sigma,T,\chi)\ll(QT)^{(7-6\sigma)/(6-4\sigma)+\varepsilon}
\end{equation}
without the advantage of the mean value estimate \eqref{quadRbound}.  After proving \eqref{quadRbound}, D. R. Heath-Brown \cite[Theorem 3]{HB2} was able to improve the $Q$-aspect of \eqref{jutrealbound} to
\begin{equation}\label{hbrealbound}
\sum_{\chi\in\mathcal{C}_2(Q)}N(\sigma,T,\chi)\ll(QT)^\varepsilon Q^{3(1-\sigma)/(2-\sigma)}T^{(3-2\sigma)/(2-\sigma)}.
\end{equation}
However, \eqref{hbrealbound} was obtained by first bounding the number of zeros in the subregions 
\[\{\rho:\sigma\leq\beta<\sigma+(\log QT)^{-1},\tau\leq\gamma<\tau+(\log QT)^{-1}\}\quad\text{with}\quad|\tau|\leq T\]
and then summing trivially over these subregions to obtain a bound for the total number of zeros in the rectangle $\{\rho:\sigma\leq\beta\leq1,|\gamma|\leq T\}$.  By considering the whole rectangle from the start and employing Lemma~\ref{Savebound} to average over the $\rho$ in the rectangle, we are able to improve the $T$-aspect of \eqref{hbrealbound} in our result \eqref{quadzeroden}.  Moreover, \eqref{quadzeroden} is an improvement of \eqref{jutrealbound} when $Q^{-4+11\sigma-6\sigma^2}\gg T^{-10+21\sigma-10\sigma^2}$, which is true for all $Q$, $T>1$ when $\sigma\geq \frac{21-\sqrt{41}}{20}\doteqdot0.7298$.\newline

We end the paper with the following remark.  Recent heuristics in \cite{DuRa} gave rise to some surprising revelations on the true optimal bound in the large sieve inequality for cubic Hecke characters, based on which, as well as its quartic analogue, the estimates in \eqref{cubicRbound} and \eqref{quarticRbound} are derived.  Thus it gives one pause in conjecturing what the best possible form of the large sieve inequality for cubic and quartic Dirichlet characters should be.  Consequently, unlike Theorem~\ref{sparse}, it is unclear what the best possible unconditional bounds one can hope for in \eqref{CSzeroden} and \eqref{quartzeroden} may be using the methods of this paper.

\vspace*{.5cm}

\noindent{\bf Acknowledgments.}  The results of this paper form part of the first-named author's honors thesis at the University of New South Wales (UNSW).  The authors were supported by the Faculty Silverstar Award PS65447 at UNSW during this work.

\bibliography{zerodenfamdirichlet12}
\bibliographystyle{amsxport}

\end{document}